\documentclass[smallextended]{svjour3}       % onecolumn (second format)
\smartqed  % flush right qed marks, e.g. at end of proof
\usepackage{graphicx}
\usepackage{amssymb}

\newcommand{\R}[0]{{\mathbb{R}}}
\newcommand{\C}[0]{{\mathbb{C}}}
\newcommand{\N}[0]{{\mathbb{N}}}

\begin{document}

\title{Optimal Newton-Secant like methods without memory for solving
nonlinear equations with its dynamics%\thanks{Grants or other notes
}

\titlerunning{Optimal Newton-Secant like methods without memory for solving
nonlinear equations with its dynamics}        % if too long for running head

\author{Mehdi Salimi$^a$\thanks{Corresponding author: mehdi.salimi@tu-dresden.de} \and Taher Lotfi$^b$ \and Somayeh Sharifi$^c$
\and Stefan Siegmund$^a$} %etc.


\institute{$^{a}$Department of Mathematics, Technische
Universit{\"a}t Dresden, 01062 Dresden, Germany\\
$^{b}$Department of Mathematics, Hamedan Branch,
Islamic Azad University, Hamedan, Iran\\
$^{c}$Young Researchers and Elite Club, Hamedan Branch, Islamic
Azad University, Hamedan, Iran}

\date{Received: date / Accepted: date}

\maketitle

\begin{abstract}
We construct two optimal Newton-Secant like iterative methods for
solving non-linear equations. The proposed classes have
convergence order four and eight and cost only three and four
function evaluations per iteration, respectively. These methods
support the Kung and Traub conjecture and possess a high
computational efficiency. The new methods are illustrated by numerical experiments
and a comparison with some existing optimal methods.
We conclude with an investigation of the basins of
attraction of the solutions in the complex plane.
\end{abstract}

\section{Introduction}
\label{intro}A main tool for solving nonlinear problems is the
approximation of simple roots $x^*$ of a nonlinear equation
$f(x^*)=0$ with a scalar function $f : D \subset \R \to \R$ which
is defined on an open interval $D$ (see e.g.\
\cite{Ostrowski,Petkovic,Petkovic1,Traub} and the references
therein). The secant method is a simple root-finding algorithm
which can be traced back to a historic precursor called ``rule of
double false position'' \cite{Papakonstantinou}. A modern way to
view the secant method would be to replace the derivative in the
Newton-Raphson method $x_{n+1} = x_n - \frac{f(x_n)}{f'(x_n)}$ by
a finite-difference approximation. The Newton-Raphson method is
one of the most widely used algorithms for finding roots. It is of
second order and requires two evaluations for each iteration step,
one evaluation of $f$ and one of $f'$. Newton-Raphson iteration is
an example of a one-point iteration, i.e.\ in each iteration step
the evaluations are taken at one point. Multiple-point methods
evaluate at several points in each iteration step and in principle
allow for a higher convergence order with a lower number of
function evaluations. Kung and Traub \cite{Kung} conjectured that
no multi-point method without memory with $k$ evaluations could
have a convergence order larger than $2^{k-1}$. A multi-point
method with convergence order $2^{k-1}$ is called optimal.

In this paper we construct two new optimal multi-point methods. We
present a two-point iteration with convergence order four which
requires two evaluations of $f$ and one evaluation of $f'$ and a
three-point iteration with convergence order eight which requires
three evaluations of $f$ and one evaluation of $f'$. Both methods
combine the Newton and Secant methods and utilize the idea of
weight functions to obtain optimality in the sense of Kung and
Traub. For an alternative construction of an optimal three-point
method with convergence order eight which also uses carefully
chosen weight functions, see \cite{LotfiEtAl2013}.

For well known two-point methods without memory one can consult
e.g.\ Jarrat \cite{Jarratt}, King \cite{King} and Ostrowski
\cite{Ostrowski}. Bi et al.\ \cite{Bi1} developed an optimal
three-point iterative method with convergence order eight. Wang
and Liu used weight functions to construct optimal three-point
methods \cite{Liu1} and \cite{Liu3} and optimal convergence order
eight was achieved by Geum and Kim \cite{Geum1} and \cite{Geum3}
utilizing parametric weight functions. Based on rational
interpolation and weight functions, Sharma et al.\ introduced two
three-point methods \cite{Sharma1,Sharma2}, see also Cordero et
al. \cite{Cordero1}-\cite{Cordero5} and Soleymani et al.\
\cite{Soleymani3}, Babajee et al. \cite{Babajee}, Thukral and
Petkovic \cite{Thukral} and for recent studies the interested
reader is referred to Chun and Lee \cite{Chun} and Petkovic et
al.\ \cite{Petkovic} and Neta \cite{Neta0} has demonstrated
methods of eight and sixteen order of convergence. Alberto et al.
\cite{Alberto} have analyzed a different anomalies in a Jarrat
family of iterative root-finding methods. In \cite{Chun1} Chun et
al. introduced weight functions with a parameter into an iteration
process to increase the order of the convergence and enhance the
behavior of the iteration process. In \cite{LS} Lotfi and Salimi
pointed to serious errors that presented in the paper entitled "A
family of optimal iterative methods with fifth and tenth order
convergence for solving nonlinear equations" as well.

The paper is organized as follows: Section 2 is devoted to the
construction and convergence analysis of a new two-point method
with optimal convergence order four and a new three-point method
with optimal convergence order eight. Computational aspects,
comparisons and dynamic behavior with other methods are
illustrated in Section 3.
\section{Development of multi-point methods}

\subsection{Optimal two-point method}
\label{sec:1}In this section we construct a new optimal two-point
class of iterative methods for solving nonlinear equations. The
Newton-Secant method is given by
\begin{equation}
  \begin{array}{lrl}\label{a1}
  y_n=x_n-\frac{f(x_n)}{f^{'}(x_n)},&\\
  [1ex]
x_{n+1}=x_{n}-\frac{f^2(x_{n})}{(f(x_n)-f(y_n))f^{'}(x_{n})},\quad
(n=0, 1, \ldots),
      \end{array}
\end{equation}
where $x_0$ is an initial approximation of $x^{*}$. The
convergence order of (\ref{a1}) is three and with three
evaluations it is not optimal. We intend to increase the order of
convergence and extend (\ref{a1}) by an additional step
\begin{equation}\label{a2}
\begin{array}{lrl}
y_n=x_n-\frac{f(x_n)}{f^{'}(x_n)},&
\\[1ex]
z_{n}=x_{n}-\frac{f^2(x_{n})}{(f(x_n)-f(y_n))f^{'}(x_{n})},&
\\[1ex]
x_{n+1}=z_n-\frac{f(z_n)}{f^{'}(x_n)}.
\end{array}
\end{equation}
Method (\ref{a2}) uses four function evaluations with order of
convergence four. Therefore, this method is not optimal. In order
to decrease the number of function evaluations, we approximate
$f(z_n)$ by an expression based on $f(x_n)$, $f(y_n)$ and
$f'(x_n)$. Taylor expansion of $f$ at $x_n$ yields
\begin{equation}\label{a3}
f(z_n)=
f(x_n)+f^{'}(x_n)(z_n-x_n)+\frac{1}{2}f^{''}(x_n)(z_n-x_n)^2+O\big((z_n-x_n)^3\big),
\end{equation}
and similarly we have
\begin{equation}\label{a4}
f(y_n)=
f(x_n)+f^{'}(x_n)(y_n-x_n)+\frac{1}{2}f^{''}(x_n)(y_n-x_n)^2+O\big((y_n-x_n)^3\big).
\end{equation}
Using Newton's method and (\ref{a4}), we obtain
\begin{equation}\label{a5}
\frac{1}{2}f^{''}(x_n)\approx
\frac{f(y_n)(f^{'}(x_n))^2}{f^2(x_n)}.
\end{equation}
According to (\ref{a2}), we have
\begin{equation}\label{a6}
z_n-x_n=-\frac{f^2(x_n)}{\left(f(x_n)-f(y_n)\right)f^{'}(x_n)}.
\end{equation}
Substituting (\ref{a5}) and (\ref{a6}) into (\ref{a3}), we obtain
\begin{equation}\label{a7}
f(z_n)\approx
f(x_n)-\frac{f^2(x_n)}{f(x_n)-f(y_n)}+\frac{f(y_n)f^2(x_n)}{\left(f(x_n)-f(y_n)\right)^2}.
\end{equation}
 Substituting (\ref{a7}) into (\ref{a2}), yields
\begin{equation}\label{a8}
\begin{array}{lrl}
y_n=x_n-\frac{f(x_n)}{f'(x_n)},&
\\[0.7ex]
z_{n}=x_{n}-\frac{f^2(x_{n})}{(f(x_n)-f(y_n))f^{'}(x_{n})},&
\\[0.7ex]
x_{n+1}=z_n-\left[
1-\frac{f(x_n)}{f(x_n)-f(y_n)}\left(1+\frac{f(y_n)}{f(x_n)-f(y_n)}\right)\right]\frac{f(x_n)}{f^{'}(x_n)}.
\end{array}
\end{equation}
Although we reduced the number of function evaluations compared to
(\ref{a2}), the convergence order of (\ref{a8}) is not yet four.
In order to increase it, we consider an appropriate weight
function, namely $\phi(t_n)$, as follows:
\begin{equation}\label{a9}
\begin{array}{lrl}
y_n=x_n-\frac{f(x_n)}{f'(x_n)},&
\\[1ex]
z_{n}=x_{n}-\frac{f^2(x_{n})}{(f(x_n)-f(y_n))f'(x_{n})},&
\\[1ex]
x_{n+1}=z_n-\left[
1-\frac{f(x_n)}{f(x_n)-f(y_n)}\left(1+\frac{f(y_n)}{f(x_n)-f(y_n)}\right)\right]\frac{f(x_n)}{f'(x_n)}\phi(t_n),
\end{array}
\end{equation}
where $t_n=\frac{f(y_n)}{f(x_n)}$. In the following theorem, we
provide sufficient conditions on the weight function $\phi(t_n)$
which imply that method (\ref{a9}) has convergence order four.

\begin{theorem}\label{theorem1}
Let $D \subseteq \R$ be an open interval, $f : D \rightarrow \R$
four times continuously differentiable and let $x^{*}\in D$ be a
simple zero of $f$. If the initial point $x_{0}$ is sufficiently
close to $x^{*}$, then the method defined by (\ref{a9}) converges
to $x^{*}$ with order at least four if the weight function $\phi :
\R \rightarrow \R$ is two times continuously differentiable and
satisfies the conditions
\[
\phi(0)=0 \quad, \quad \phi^{'}(0)=-\frac{1}{2} \quad \textup{and}
\quad |\phi^{''}(0)|<\infty.
\]
\end{theorem}

\begin{proof}
 Let $e_{n}:=x_{n}-x^{*}$, $e_{n,y}:=y_{n}-x^{*}$,
$e_{n,z}:=z_{n}-x^{*}$ and
$c_{n}:=\frac{f^{(n)}(x^{*})}{n!f^{'}(x^{*})}$ for $n\in \N$.
Using the fact that $f(x^{*})=0$, Taylor expansion of $f$ at
$x^{*}$ yields
\begin{equation}\label{a10}
f(x_{n}) = f^{'}(x^*)\left(e_{n} +
c_{2}e_{n}^{2}+c_{3}e_{n}^{3}+c_{4}e_{n}^{4}\right)+O(e_n^{5})
\end{equation}
and
\begin{equation}\label{a11}
f^{'}(x_{n}) = f^{'}(x^*)\left(1 +
2c_{2}e_{n}+3c_{3}e_{n}^{2}+4c_{4}e_{n}^{3}\right)+O(e_n^{4}).
\end{equation}
Therefore
\[
\frac{f(x_{n})}{f'(x_{n})}=e_{n}-c_{2}e_{n}^{2}+\left(2c_{2}^{2}-2c_{3}\right)e_{n}^{3}+O(e_n^{4}),
\]
and hence
\[
e_{n,y}= y_n-x^*=c_{2}e_{n}^{2}+ O(e_n^3).
\]
For $f(y_n)$ we also have
\begin{equation}\label{a12}
f(y_{n}) = f^{'}(x^*)\left(c_2e_{n}^2
+(-2c_2^2+2c_3)e_n^3+(5c_2^3-7c_2c_3+3c_4)e_n^4\right)+O(e_n^{5}),
\end{equation}
therefore, by substituting (\ref{a10}), (\ref{a11}) and
(\ref{a12}) into (\ref{a2}), we get
\[
e_{n,z}= z_n-x^*=c_{2}^2e_{n}^{3}+ O(e_n^4).
\]
From (\ref{a10}) and (\ref{a12}), we obtain
\begin{equation}\label{a13}
t_n=
\frac{f(y_n)}{f(x_n)}=c_2e_n+(-3c_2^2+2c_3)e_n^2+(8c_2^3-10c_2c_3+3c_4)e_n^3+O(e_n^{4}).
\end{equation}
Expanding $\phi$ at $0$, yields
\begin{equation}\label{a14}
\phi(t_n)=\phi(0)+\phi^{'}(0)t_n+\frac{1}{2}\phi^{''}(0)t_n^2+O(t_n^{3}).
\end{equation}
Substituting (\ref{a10})-(\ref{a14}) into (\ref{a9}), we obtain
\[
e_{n+1}=x_{n+1}-x^*=R_2e_n^2+R_3e_n^3+R_4e_n^4+O(e_n^5),
\]
where
\begin{equation}
\begin{array}{lrl}
R_2=2c_2 \phi(0),&
\\[1ex]
R_3=c_2^2\big(1+2\phi^{'}(0)\big),&
\\[1ex]
R_4=-c_2c_3+c_2^3\big(\frac{5}{2}+\phi^{''}(0)\big).
\end{array}
\end{equation}
By setting $R_2=R_3=0$, the convergence order becomes four.
Obviously
\begin{equation}
\begin{array}{lrl}
\phi(0)=0 \quad \Rightarrow \quad R_2=0,&
\\[1ex]
\phi^{'}(0)=-\frac{1}{2} \quad \Rightarrow \quad R_3=0,
\\[1ex]
|\phi^{''}(0)|<\infty \quad \Rightarrow \quad R_4\neq0.
\end{array}
\end{equation}
Consequently,  the error equation becomes
\[
e_{n+1}=R_{4}e_n^4+O(e_n^5),
\]
which finishes the proof of the theorem.
\end{proof}
\subsection{Optimal three-point method}
In this section we construct a new optimal three-point method
based on the two-point method (\ref{a9}). We extend method
(\ref{a9}) by a Newton step and get
\begin{equation}\label{a15}
\begin{array}{lrl}
y_n=x_n-\frac{f(x_n)}{f^{'}(x_n)},&
\\[0.7ex]
z_{n}=x_{n}-\frac{f^2(x_{n})}{(f(x_n)-f(y_n))f'(x_{n})},&
\\[0.7ex]
v_n=z_n-\left[
1-\frac{f(x_n)}{f(x_n)-f(y_n)}\left(1+\frac{f(y_n)}{f(x_n)-f(y_n)}\right)\right]\frac{f(x_n)}{f'(x_n)}\phi(t_n),\\[0.7ex]
x_{n+1}=v_n-\frac{f(v_n)}{f^{'}(v_n)},
\end{array}
\end{equation}
where $\phi(t_n)$ is a weight function as in Theorem
\ref{theorem1}.

Method (\ref{a15}) evaluates functions for five times with order
of convergence eight, so the method is not optimal. In order to
reduce the number of function evaluation, we approximate $f'(v_n)$
by an expression which is based on $f(x_n)$, $f(y_n)$, $f(v_n)$,
and $f'(x_n) $, namely its linear approximation
\begin{equation}\label{a16}
f^{'}(v_n)\approx
f^{'}(x_n)+\frac{f^{'}(z_n)-f^{'}(x_n)}{z_n-x_n}(v_n-x_n).
\end{equation}
We approximate $f'(z_n)$ by expressions which were calculated
above. The Taylor expansion of $f$ at $y_n$ yields
\begin{equation}\label{a17}
f(z_n)=
f(y_n)+f^{'}(y_n)(z_n-y_n)+\frac{1}{2}f^{''}(y_n)(z_n-y_n)^2+O\big((z_n-y_n)^3\big),
\end{equation}
and
\begin{equation}\label{a18}
f^{'}(z_n)=
f^{'}(y_n)+f^{''}(y_n)(z_n-y_n)+O\big((z_n-y_n)^2\big).
\end{equation}
According to (\ref{a17}), we have
\begin{equation}\label{a19}
f^{'}(y_n)\approx
\frac{f(z_n)-f(y_n)}{z_n-y_n}-\frac{1}{2}f^{''}(y_n)(z_n-y_n).
\end{equation}
On the other hand, we have
\begin{equation}\label{a20}
f^{''}(y_n)\approx
2f[z_n,x_n,x_n]=\frac{2\left(f[z_n,x_n]-f^{'}(x_n)\right)}{z_n-x_n},
\end{equation}
where $f[z_n, x_n]=\frac{f(z_n)-f(x_n)}{z_n-x_n}$. Substituting
(\ref{a19}) and (\ref{a20}) into (\ref{a18}), we obtain
\begin{equation}\label{a21}
f^{'}(z_n)\approx
f[z_n,y_n]+\big(f[z_n,x_n]-f^{'}(x_n)\big)\frac{z_n-y_n}{z_n-x_n},
\end{equation}
where $f[z_n, y_n]=\frac{f(z_n)-f(y_n)}{z_n-y_n}$. In a next step
we replace $f(z_n)$ by an approximation to reduce the number of
function evaluations. Taylor expansion of $f$ at $x_n$ yields
\begin{equation}\label{a22}
\begin{array}{lrl}
f(z_{n}) =
f(x_n)+f^{'}(x_n)(z_n-x_n)+\frac{1}{2}f^{''}(x_n)(z_n-x_n)^2+\frac{1}{6}f^{'''}(x_n)(z_n-x_n)^3
& \\[1ex]
\quad \quad \quad +O\big((z_n-x_n)^4\big),
\end{array}
\end{equation}
and similarly we have
\begin{equation}\label{a23}
\begin{array}{lrl}
f(v_{n}) =
f(x_n)+f^{'}(x_n)(v_n-x_n)+\frac{1}{2}f^{''}(x_n)(v_n-x_n)^2+\frac{1}{6}f^{'''}(x_n)(v_n-x_n)^3
& \\[1ex]
\quad \quad \quad +O\big((v_n-x_n)^4\big).
\end{array}
\end{equation}
From (\ref{a23}), we calculate
\begin{equation}\label{a24}
\frac{1}{6}f^{'''}(x_n) \approx
\left[\frac{f(v_n)-f(x_n)}{v_n-x_n}-f^{'}(x_n)-\frac{f(y_n)\left(f^{'}(x_n)\right)^2}{f^2(x_n)}(v_n-x_n)\right]\frac{1}{(v_n-x_n)^2}.
\end{equation}
Plugging (\ref{a5}) and (\ref{a24}) into (\ref{a22}), we obtain
\begin{equation}\label{a25}
\begin{array}{lrl}
f(z_{n}) \approx
f(x_n)+f^{'}(x_n)(z_n-x_n)+\frac{f(y_n)\left(f^{'}(x_n)\right)^2}{f^2(x_n)}(z_n-x_n)^2&
\\[1ex]
+\left[f[v_n,x_n]-f^{'}(x_n)-\frac{f(y_n)\left(f^{'}(x_n)\right)^2}{f^2(x_n)}(v_n-x_n)\right]\frac{(z_n-x_n)^3}{(v_n-x_n)^2}.
\end{array}
\end{equation}
Then, by replacing (\ref{a21}) into (\ref{a16}), we get
\begin{equation}\label{a226}
f^{'}(v_n)\approx
f^{'}(x_n)+\frac{f[z_n,y_n]+\left(f[z_n,x_n]-f^{'}(x_n)\right)\frac{z_n-y_n}{z_n-x_n}-f^{'}(x_n)}{z_n-x_n}(v_n-x_n),
\end{equation}
where we can plug (\ref{a25}) instead of $f(z_n)$ in (\ref{a226})
as well. The following scheme evaluates functions for four times
\begin{equation}\label{a26}
\begin{array}{lrl}
y_n=x_n-\frac{f(x_n)}{f^{'}(x_n)},&
\\[1ex]
z_{n}=x_{n}-\frac{f^2(x_{n})}{(f(x_n)-f(y_n))f'(x_{n})},&
\\[1ex]
v_n=z_n-\left[
1-\frac{f(x_n)}{f(x_n)-f(y_n)}\left(1+\frac{f(y_n)}{f(x_n)-f(y_n)}\right)\right]\frac{f(x_n)}{f'(x_n)}\phi(t_n),&
\\[1ex]
x_{n+1}=v_n-f(v_n)\left(f^{'}(x_n)+\frac{f[z_n,y_n]+\left(f[z_n,x_n]-f^{'}(x_n)\right)\frac{z_n-y_n}{z_n-x_n}-f^{'}(x_n)}{z_n-x_n}(v_n-x_n)\right)^{-1},
\end{array}
\end{equation}
where $f(z_n)$ is evaluated  from (\ref{a25}) and
$t_n=\frac{f(y_n)}{f(x_n)}$.

Method (\ref{a26}) is not still optimal. Therefore we introduce a
second weight function as follows:
\begin{equation}\label{a27}
\begin{array}{lrl}
y_n=x_n-\frac{f(x_n)}{f'(x_n)},&
\\[1ex]
z_{n}=x_{n}-\frac{f^2(x_{n})}{(f(x_n)-f(y_n))f'(x_{n})},&
\\[1ex]
v_n=z_n-\left[
1-\frac{f(x_n)}{f(x_n)-f(y_n)}\left(1+\frac{f(y_n)}{f(x_n)-f(y_n)}\right)\right]\frac{f(x_n)}{f'(x_n)}\phi(t_n),&
\\[1ex]
x_{n+1}=v_n-f(v_n)
& \\[1ex]
\quad \quad \quad
\left(f^{'}(x_n)+\frac{f[z_n,y_n]+\left(f[z_n,x_n]-f^{'}(x_n)\right)\frac{z_n-y_n}{z_n-x_n}-f^{'}(x_n)}{z_n-x_n}(v_n-x_n)\right)^{-1}\psi(s_n),
\end{array}
\end{equation}
where $f(z_n)$ is evaluated  from (\ref{a25}) and
$t_n=\frac{f(y_n)}{f(x_n)}$ and $s_n=\frac{f(v_n)}{f(x_n)}$.\\

In the following theorem we prove that method (\ref{a27}) is of
convergence order eight if the weight functions $\phi(t_n)$ and
$\psi(s_n)$ satisfy the stated conditions in the following
theorem.
\begin{theorem}\label{theorem2}
Let $D \subseteq \R$ be an open interval, $f : D \rightarrow \R$
eight times continuously differentiable and let $x^{*}\in D$ be a
simple zero of $f$. If the initial point $x_{0}$ is sufficiently
close to $x^{*}$, then the method defined by (\ref{a27}) converges
to $x^{*}$ with order at least eight if the weight function $\phi
: \R \rightarrow \R$ is two times continuously differentiable,
$\psi : \mathbb{R} \rightarrow \R$ is continuously differentiable
and they satisfy the conditions of Theorem \ref{theorem1} and
moreover
\[
\phi^{''}(0)=-\frac{5}{2}, \quad \psi(0)=1 \quad \textup{and}
\quad \psi^{'}(0)=1.
\]
\end{theorem}
\begin{proof}
 Let $e_{n}:=x_{n}-x^{*}$, $e_{n,y}:=y_{n}-x^{*}$,
$e_{n,z}:=z_{n}-x^{*}$,
$c_{n}:=\frac{f^{(n)}(x^{*})}{n!f^{'}(x^{*})}$ for $n \in \N$.
Using the fact that $f(x^*)=0$, Taylor expansion of $f$ at $x^{*}$
yields
\begin{equation}\label{a28}
f(x_{n}) = f^{'}(x^*)(e_{n} +
c_{2}e_{n}^{2}+c_{3}e_{n}^{3}+c_{4}e_{n}^{4}+c_{5}e_{n}^{5}+c_{6}e_{n}^{6}+c_{7}e_{n}^{7}+c_8e_{n}^{8})+O(e_{n}^{9}),
\end{equation}
and
\begin{equation}\label{a29}
f^{'}(x_{n})=f^{'}(x^{*})(1+2c_{2}e_{n}+3c_{3}e_{n}^{2}+\dots+9c_{9}e_{n}^{8})+O(e_{n}^{9}).
\end{equation}
According to  Theorem \ref{theorem1}, we get
\[
e_{n,y}=y_n-x^{*}=c_2e_n^2+(-2c_2^2+2c_3)e_n^3+(4c_2^3-7c_2c_3+3c_4)e_n^4+O(e_n^5),
\]
and
\[
e_{n,v}=v_n-x^{*}=\big(
\big(\frac{5}{2}+\phi^{''}(0)\big)c_2^3-c_2c_3\big)e_n^4+O(e_n^5).
\]
By using Taylor's theorem for $f(y_n)$ and $f(v_n)$ at $x^{*}$, we
have
\begin{equation}\label{a30}
\begin{array}{lrl}
f(y_{n}) =
f^{'}(x^{*})\big[c_2e_n^2-2(c_2^2-c_3)e_n^3+(5c_2^3-7c_2c_3+3c_4)e_n^4
&
\\[1ex]
- 2(6c_2^4-12c_2^2c_3+3c_3^2+5c_2c_4-2c_5)e_n^5 &
\\[1ex]
+
\left(28c_2^5-73c_2^3c_3+37c_2c_3^2+34c_2^2c_4-17c_3c_4-13c_2c_5+5c_6
\right)e_n^6\big]+O(e_n^7),
\end{array}
\end{equation}
and
\begin{equation}\label{a31}
\begin{array}{lrl}
f(v_{n}) =
f^{'}(x^{*})\big[c_2c_3e_n^4+\frac{1}{4}\left(29c_2^4+8c_2^2c_3-8c_3^2\right)e_n^5
&
\\[1ex]
+\left(-60.75c_2^5+54c_2^3c_3+6c_2c_3^2+3c_2^2c_4-7c_3c_4-3c_2c_5\right)e_n^6
&
\\[1ex]
\frac{1}{4}(1243c_2^6-2166c_2^4c_3+332c_2^3c_4+8c_2^2(77c_3^2+2c_5)
&
\\[1ex]
+8(2c_3^3-3c_4^2-5c_3c_5)+16c_2(4c_3^2c_4-c_6))e_n^7\big]+O(e_n^8).
\end{array}
\end{equation}
Also
\begin{equation}\label{a32}
\begin{array}{lrl}
f(z_{n}) =
f^{'}(x^{*})\big[c_2^2e_n^3+3c_2(-c_2^2+c_3)e_n^4+(6c_2^4-13c_2^2c_3+2c_3^2+4c_2c_4)e_n^5 &
\\[1ex]
+(-8c_2^5+33c_2^3c_3-18c_2c_3^2-18c_2^2c_4+5c_3c_4+5c_2c_5)e_n^6 &
\\[1ex]
(3c_2^6-\frac{175}{4}c_2^4c_3+48c_2^3c_4+c_2^2(64c_3^2-23c_5) &
\\[1ex]
+(-8c_3^3+3c_4^2+6c_3c_5)+c_2(-50c_3c_4))e_n^7\big]+O(e_n^8).
\end{array}
\end{equation}
Moreover, for $f^{'}(v_n)$, we also have
\begin{equation}\label{a33}
\begin{array}{lrl}
f^{'}(v_{n}) =
f^{'}(x^{*})\big[1-c_2c_3e_n^3+\left(c_2^2c_3-2c_3^2-c_2c_4\right)e_n^4 &
\\[1ex]
+\frac{1}{4}\left(48c_2^5-197c_2^3c_3+104c_2c_3^2+60c_2^2c_4-20c_3c_4-4c_2c_5\right)e_n^5
&
\\[1ex]
-\frac{1}{2}(243c_2^6-412c_2^4c_3-16c_3^3+39c_2^3c_4+6c_4^2+12c_3c_5
&
\\[1ex]
+c_2^2(165c_3^2+4c_5)+c_2(-20c_3c_4+2c_6))e_n^6\big]+O(e_n^7).
\end{array}
\end{equation}
From (\ref{a28}) and (\ref{a31}), we calculate
\begin{equation}\label{a34}
\begin{array}{lrl}
s_n=\frac{f(v_n)}{f(x_n)}=\left(-c_2c_3\right)e_n^3+\left(7.25c_2^4+3c_2^2c_3-2c_3^2-2c_2c_4\right)e_n^4 &
\\[1ex]
+\left(-68c_2^5+51c_2^3c_3+5c_2^2c_4-7c_3c_4+9c_2c_3^2-3c_2c_5\right)e_n^5+O(e_n^6).
\end{array}
\end{equation}
Expanding $\psi$ at $0$, yields
\begin{equation}\label{a35}
\psi(s_n)=\psi(0)+\psi^{'}(0)s_n+\frac{1}{2}\psi^{''}(0)s_n^2+O(s_n^3).
\end{equation}
By substituting (\ref{a28})-(\ref{a35}) into (\ref{a27}), we
obtain
\[
e_{n+1}=x_{n+1}-x^{*}=R_4e_n^4+R_5e_n^5+R_6e_n^6+R_7e_n^7+R_8e_n^8+O(e_n^9),
\]
where
\begin{equation}
\begin{array}{lrl}
R_4=-\frac{1}{2}c_2(-1+\psi(0))\big(\big(5+2\phi^{''}(0)\big)c_2^2-2c_3\big),
&
\\[1ex]
R_5=0, &
\\[1ex]
R_6=0, &
\\[1ex]
R_7=-\frac{1}{4}c_2^2\big(-2c_3+c_2^2\big(5+2\phi^{''}(0)\big)\big)
& \\[1ex]
\quad \quad
\big(-2c_3\big(-1+\psi^{'}(0)\big)+c_2^2\big(5+2\phi^{''}(0)\big)\psi^{'}(0)\big),
&
\\[1ex]
R_8=\frac{1}{4}c_2^2c_3\left(29c_2^3+4c_2c_3-4c_4\right).
\end{array}
\end{equation}
To ensure convergence order eight for the three-point method
(\ref{a27}), it is necessary to have $R_i=0$, $(i=4,5,6,7)$.
Obviously
\begin{equation}
\begin{array}{lrl}
\psi(0)=1 \quad \Rightarrow \quad R_4=0, &
\\[1ex]
\psi^{'}(0)=1, \phi^{''}(0)=-\frac{5}{2}\quad \Rightarrow \quad
R_7&=0.
\end{array}
\end{equation}

It is clear that $R_{8}\neq 0$, thus the error equation becomes
\[
e_{n+1}=R_{8}e_n^8+O(e_n^9),
\]
and method (\ref{a27}) has convergence order eight, which proves
the theorem.
\end{proof}
In what follows, we give some concrete explicit representations of
(\ref{a27}) by choosing different weight functions satisfying the
provided condition for the weight functions $\phi(t_n)$ and
$\psi(t_n)$ in Theorems 1 and 2.

 \textbf{Method
1.} Choose the weight functions $\phi(t_n)$ and $\psi(s_n)$
 as follows:
\begin{equation}\label{a36}
\phi(t_n)=-\frac{1}{2}t_n-\frac{5}{4}t_n^2 \quad \textup{and}
\quad \psi(s_n)=\frac{1+2s_n}{1+s_n},
\end{equation}
where $t_n=\frac{f(y_n)}{f(x_n)}$ and $s_n=\frac{f(v_n)}{f(x_n)}$.
The functions  $\phi(t_n)$ and $\psi(s_n)$  in (\ref{a36}) satisfy
the assumptions of Theorem \ref{theorem2} denoted by SLSS, so
\begin{equation}\label{a111}
\begin{array}{lrl}
y_n=x_n-\frac{f(x_n)}{f^{'}(x_n)}, &
\\[1ex]
z_{n}=x_{n}-\frac{f^2(x_{n})}{(f(x_n)-f(y_n))f'(x_{n})}, &
\\[1ex]
v_n=z_n-\left[
1-\frac{f(x_n)}{f(x_n)-f(y_n)}\left(1+\frac{f(y_n)}{f(x_n)-f(y_n)}\right)\right]\left(\frac{-f(y_n)}{2f(x_n)}-\frac{5}{4}\left(\frac{f(y_n)}{f(x_n)}\right)^2\right)\frac{f(x_n)}{f'(x_n)}, &
\\[1ex]
x_{n+1}=v_n-f(v_n)\left(\frac{f(x_n)+2
f(v_n)}{f(x_n)+f(v_n)}\right)\times &
\\[1ex]
\quad \quad \quad
\left(f^{'}(x_n)+\frac{f[z_n,y_n]+\left(f[z_n,x_n]-f^{'}(x_n)\right)\frac{z_n-y_n}{z_n-x_n}-f^{'}(x_n)}{z_n-x_n}(v_n-x_n)\right)^{-1},
\end{array}
\end{equation}
where $f(z_n)$ is evaluated by (\ref{a25}).

\textbf{Method 2.} Choose the weight functions $\phi(t_n)$ and
$\psi(s_n)$  as follows:
\begin{equation}\label{a37}
\phi(t_n)=t_n+\frac{9t_n}{5t_n-6} \quad \textup{and} \quad
\psi(s_n)=\frac{1}{1-s_n},
\end{equation}
where $t_n=\frac{f(y_n)}{f(x_n)}$ and $s_n=\frac{f(v_n)}{f(x_n)}$.
The functions $\phi(t_n)$ and $\psi(s_n)$  in (\ref{a37}) satisfy
the assumptions of Theorem \ref{theorem2} and we get
\begin{equation}\label{a112}
\begin{array}{lrl}
y_n=x_n-\frac{f(x_n)}{f^{'}(x_n)}, &
\\[1ex]
z_{n}=x_{n}-\frac{f^2(x_{n})}{(f(x_n)-f(y_n))f'(x_{n})}, &
\\[1ex]
v_n=z_n-\left[
1-\frac{f(x_n)}{f(x_n)-f(y_n)}\left(1+\frac{f(y_n)}{f(x_n)-f(y_n)}\right)\right]\left(\frac{f(y_n)}{f(x_n)}+\frac{9f(y_n)}{5f(y_n)-6f(x_n)}\right)\frac{f(x_n)}{f'(x_n)}, &
\\[1ex]
x_{n+1}=v_n-f(v_n)\left(\frac{f(x_n)}{f(x_n)-f(v_n)}\right)\times
&
\\[1ex]
\quad \quad
\quad\left(f^{'}(x_n)+\frac{f[z_n,y_n]+\left(f[z_n,x_n]-f^{'}(x_n)\right)\frac{z_n-y_n}{z_n-x_n}-f^{'}(x_n)}{z_n-x_n}(v_n-x_n)\right)^{-1},
\end{array}
\end{equation}
where $f(z_n)$ is evaluated by (\ref{a25}).

\textbf{Method 3.} Choose the weight functions $\phi(t_n)$ and
$\psi(s_n)$  as follows:
\begin{equation}\label{a38}
\phi(t_n)=\frac{t_n}{5t_n-2} \quad \textup{and} \quad
\psi(s_n)=1+\frac{2s_n}{2+5s_n},
\end{equation}
where $t_n=\frac{f(y_n)}{f(x_n)}$ and $s_n=\frac{f(v_n)}{f(x_n)}$.
The functions $\phi(t_n)$ and $\psi(s_n)$ in (\ref{a38}) satisfy
the assumptions of Theorem \ref{theorem2} and we get
\begin{equation}\label{a113}
\begin{array}{lrl}
y_n=x_n-\frac{f(x_n)}{f^{'}(x_n)}, &
\\[1ex]
z_{n}=x_{n}-\frac{f^2(x_{n})}{(f(x_n)-f(y_n))f'(x_{n})}, &
\\[1ex]
v_n=z_n-\left[
1-\frac{f(x_n)}{f(x_n)-f(y_n)}\left(1+\frac{f(y_n)}{f(x_n)-f(y_n)}\right)\right]\left(\frac{f(y_n)}{5f(y_n)-2f(x_n)}\right)\frac{f(x_n)}{f'(x_n)}, &
\\[1ex]
x_{n+1}=v_n-f(v_n)\left(1+\frac{2f(v_n)}{2f(x_n)+5f(v_n)}\right)\times
&
\\[1ex]
\quad \quad
\quad\left(f^{'}(x_n)+\frac{f[z_n,y_n]+\left(f[z_n,x_n]-f^{'}(x_n)\right)\frac{z_n-y_n}{z_n-x_n}-f^{'}(x_n)}{z_n-x_n}(v_n-x_n)\right)^{-1},
\end{array}
\end{equation}
where $f(z_n)$ is evaluated by (\ref{a25}).

\textbf{Method 4.} Choose the weight functions $\phi(t_n)$ and
$\psi(s_n)$  as follows:
\begin{equation}\label{a39}
\phi(t_n)=-\frac{6t_n+t_n^2}{4}+\frac{t_n}{1+t_n} \quad
\textup{and} \quad \psi(s_n)=(1+s_n)^{\frac{s_n+1}{2s_n+1}},
\end{equation}
where $t_n=\frac{f(y_n)}{f(x_n)}$ and $s_n=\frac{f(v_n)}{f(x_n)}$.
The functions $\phi(t_n)$ and $\psi(s_n)$  in (\ref{a39}) satisfy
the assumptions of Theorem \ref{theorem2} and we get
\begin{equation}\label{a114}
\begin{array}{lrl}
y_n=x_n-\frac{f(x_n)}{f^{'}(x_n)}, &
\\[1ex]
z_{n}=x_{n}-\frac{f^2(x_{n})}{(f(x_n)-f(y_n))f'(x_{n})}, &
\\[1ex]
v_n=z_n-\left[
1-\frac{f(x_n)}{f(x_n)-f(y_n)}\left(1+\frac{f(y_n)}{f(x_n)-f(y_n)}\right)\right]\times &
\\[1ex]
\quad \quad
\left(\frac{-6f(y_n)}{4f(x_n)}-\frac{f^2(y_n)}{4f^2(x_n)}+\frac{f(y_n)}{f(x_n)+f(y_n)}\right)\frac{f(x_n)}{f'(x_n)},
&
\\[1ex]
x_{n+1}=v_n-f(v_n)\left(\left(1+\frac{f(v_n)}{f(x_n)}\right)^{\frac{f(v_n)+f(x_n)}{2f(v_n)+f(x_n)}}\right)\times\\
\quad \quad \quad
\left(f^{'}(x_n)+\frac{f[z_n,y_n]+\left(f[z_n,x_n]-f^{'}(x_n)\right)\frac{z_n-y_n}{z_n-x_n}-f^{'}(x_n)}{z_n-x_n}(v_n-x_n)\right)^{-1},
\end{array}
\end{equation}
where $f(z_n)$ is evaluated by (\ref{a25}).

In the next section we apply the new methods (\ref{a111}),
(\ref{a112}), (\ref{a113}) and (\ref{a114}) to several benchmark
examples and compare them with existing three-point methods which
have the same order of convergence and the same computational
efficiency index equal to $\sqrt[\theta]{r} = 1.682$ for the
convergence order $r=8$ which is optimal for $\theta = 4$ function
evaluations per iteration \cite{Ostrowski,Traub}.
\\
\section{Numerical performance and dynamic behavior}
\subsection{Numerical results }
In this section we test and compare our proposed methods with some existing methods.
 We compare our Methods 1-4 with the following related three-point methods.
\\
\\
\textbf{W. Bi, H. Ren and Q. Wu method.} The method by Bi
et al.\ \cite{Bi1} denoted by BRW is
\begin{equation}\label{b3}
\begin{array}{lrl}
y_{n}=x_{n}-\frac{f(x_{n})}{f'(x_{n})}, &
\\[1ex]
z_{n}=y_{n}-\frac{f(y_{n})}{f'(x_{n})}\cdot\frac{f(x_n)+\beta f(y_n)}{f(x_n)+(\beta-2)f(y_n)} , &
\\[1ex]
x_{n+1}=z_{n}- \frac{f(z_n)}{f[z_n, y_n]+f[z_n, x_n,
x_n](z_n-y_n)}H(t_n),
\end{array}
\end{equation}
with weight function
\begin{equation}\label{b4}
H(t_n)=\frac{1}{(1-\alpha t_n)^2}, \quad \alpha=1,
\end{equation}
and $t_n=\frac{f(z_n)}{f(x_n)}$ and $\beta=-\frac{1}{2}$.
\\
\\
\\
\textbf{Wang and  Liu method.} The method by Wang and  Liu
\cite{Liu3} denoted by WL is
\begin{equation}\label{b5}
\begin{array}{lrl}
y_n=x_n-\frac{f(x_n)}{f'(x_n)}, &
\\[1ex]
z_n=x_n-\frac{f(x_n)}{f'(x_n)}~G(t_n), &
\\[1ex]
x_{n+1}=z_n-\frac{f(z_n)}{f'(x_n)}\left(H(t_n)+V(t_n)W(s_n)\right),
\end{array}
\end{equation}
with weight functions
\begin{equation}\label{b6}
G(t_n)=\frac{1-t_n}{1-2t_n}, \quad
H(t_n)=\frac{5-2t_n+t_n^2}{5-12t_n}, \quad V(t_n)=1+4 t_n, \quad
W(s_n)=s_n,
\end{equation}
and $t_n=\frac{f(y_n)}{f(x_n)}$ and $s_n=\frac{f(z_n)}{f(y_n)}$.
\\
\\
\textbf{Sharma and Sharma method.} The Sharma and  Sharma method
\cite{Sharma1} denoted by SS is
\begin{equation}\label{b7}
\begin{array}{lrl}
y_n=x_n-\frac{f(x_n)}{f'(x_n)}, &
  \\[1ex]
   z_n=y_n-\frac{f(y_n)}{f'(x_n)}\cdot\frac{f(x_n)}{f(x_n)-2f(y_n)},
&
  \\[1ex]
x_{n+1}=z_n-\frac{f[x_n,y_n]f(z_n)}{f[x_n,z_n]f[y_n,z_n]}~W(t_n),
\end{array}
\end{equation}
where weight functions are
\begin{equation}\label{b8}
W(t_n)=1+\frac{t_n}{1+\alpha t_n}, \quad \alpha=1,
\end{equation}
and $t_n=\frac{f(z_n)}{f(x_n)}$.
\\
\\
\textbf{Babajee et al. method.} The method by Babajee et al.,  see
\cite{Babajee}, denoted by BCST,  is
\begin{equation}\label{b9}
  \begin{array}{lrl}
    y_n=x_{n}-\frac{f(x_{n})}{f'(x_{n})} \cdot\left(1+\left(\frac{f(x_n)}{f'(x_n)}\right)^5\right), &
  \\[1ex]
    z_n=y_{n}-\frac{f(y_n)}{f'(x_n)}\cdot\left(1-\frac{f(y_n)}{f(x_n)}\right)^{-2} ,&
  \\[1ex]
  x_{n+1}=z_{n}-
    \frac{f(z_n)}{f'(x_n)}\cdot\frac{1+\left(\frac{f(y_n)}{f(x_n)}\right)^2+5\left(\frac{f(y_n)}{f(x_n)}\right)^4+\frac{f(z_n)}{f(y_n)}}{\left(1-\frac{f(y_n)}{f(x_n)}-\frac{f(z_n)}{f(x_n)}\right)^2}.&
      \end{array}
\end{equation}
\textbf{Cordero et al. method.} The method by Cordero et al., see
\cite{Cordero5}, denoted by CFGT,  is

\begin{equation}
  \begin{array}{lrl}\label{b10}
    y_n=x_{n}-\frac{f(x_{n})}{f'(x_{n})}, &
  \\[1ex]
    z_n=y_{n}-\frac{f^3(x_n)}{f^3(x_n)-2f^2(x_n)f(y_n)-f(x_n)f^2(y_n)-\frac{1}{2}f^3(y_n)}
   \cdot\frac{f(y_n)}{f'(x_n)},&
  \\[1ex]
  x_{n+1}=z_{n}-\frac{f(x_n)+3f(z_n)}{f(x_n)+f(z_n)}\cdot\frac{f(z_n)}{f[z_n,y_n]+f[z_n,x_n,x_n](z_n-y_n)},&
      \end{array}
\end{equation}
\\
with the divided differences $f[z_n,
y_n]=\frac{f(z_n)-f(y_n)}{z_n-y_n}$, $f[z_n, x_n,
x_n]=\frac{f[z_n, x_n]-f^{'}(x_n)}{z_n-x_n}$. \\
\\
\\
\textbf{Cordero et al. method.} The method by Cordero et al., see
\cite{Cordero22}, denoted by CTV,  is

\begin{equation}
  \begin{array}{lrl}\label{b11}
    y_n=x_{n}-\frac{f(x_{n})}{f'(x_{n})}, &
  \\[1ex]
    z_n=y_{n}-\frac{f(y_n)}{f'(x_n)}\cdot\frac{f(x_n)}{f(x_n)-2f(y_n)},&
  \\[1ex]
  x_{n+1}=v_{n}-\frac{f(z_n)}{f'(x_n)}\cdot\frac{\gamma(v_n-z_n)}{\beta_1(v_n-z_n)+\beta_2(y_n-x_n)+\beta_3(z_n-x_n)},&
      \end{array}
\end{equation}
\\
where
\[
v_n=z_n-\frac{f(z_n)}{f'(x_n)}\cdot\left(\frac{f(x_n)-f(y_n)}{f(x_n)-2f(y_n)}+\frac{1}{2}\frac{f(z_n)}{f(y_n)-2f(z_n)}\right)^2,
\]
and $\gamma, \beta_1, \beta_2, \beta_3\in \mathbb{R}$ such that
$\gamma=3(\beta_2+\beta_3)\neq 0$.\\
\\
\\
\textbf{Thukral and Petkovic method.} The method by Thukral and
Petkovic., see \cite{Thukral}, denoted by TP,  is

\begin{equation}\label{b12}
  \begin{array}{lrl}
    y_n=x_{n}-\frac{f(x_{n})}{f'(x_{n})}, &
  \\[1ex]
    z_n=y_{n}-\frac{f(y_n)}{f'(x_n)}\cdot\frac{f(x_n)+\beta f(y_n)}{f(x_n)+(\beta-2)f(y_n)},\quad \quad (\alpha,\beta\in \R)&
  \\[1ex]
  x_{n+1}=z_{n}-\frac{f(z_n)}{f'(x_n)}\cdot\left(H(t_n)+\frac{f(z_n)}{f(y_n)-\alpha f(z_n)}+\frac{4f(z_n)}{f(x_n)}\right),&
      \end{array}
\end{equation}
\\
with weight functions
\begin{equation}\label{b13}
H(t_n)=\frac{5-2\beta
-(2-8\beta+2\beta^2)t_n+(1+4\beta)t_n^2}{5-2\beta-(12-12\beta+2\beta^2)t_n},
\end{equation}
where $t_n=\frac{f(y_n)}{f(x_n)}$.\\
\\
\\
\textbf{Chun and Lee method.} The method by Chun and Lee., see
\cite{Chun}, denoted by CL,  is

\begin{equation}
  \begin{array}{lrl}\label{b14}
    y_n=x_{n}-\frac{f(x_{n})}{f'(x_{n})}, &
  \\[1ex]
    z_n=y_{n}-\frac{f(y_n)}{f'(x_n)}\cdot\frac{1}{\left(1-\frac{f(y_n)}{f(x_n)}\right)^2},&
  \\[1ex]
  x_{n+1}=z_{n}-\frac{f(z_n)}{f'(x_n)}\cdot\frac{1}{\left(1-H(t_n)-J(s_n)-P(u_n)\right)^2},&
      \end{array}
\end{equation}
\\
with weight functions
\begin{equation}\label{b15}
H(t_n)=-\beta-\gamma+t_n+\frac{t_n^2}{2}-\frac{t_n^3}{2}, \quad
J(s_n)=\beta+\frac{s_n}{2}, \quad P(u_n)=\gamma+\frac{u_n}{2},
\end{equation}
where $t_n=\frac{f(y_n)}{f(x_n)}$, $s_n=\frac{f(z_n)}{f(x_n)}$, $u_n=\frac{f(z_n)}{f(y_n)}$ and $\beta, \gamma \in\R$.\\
\\
\\
The three-point method (\ref{a27}), more precisely, the explicitly
proposed methods (\ref{a111}), (\ref{a112}), (\ref{a113}) and
(\ref{a114}), are now tested on a number of nonlinear equations.
To obtain a high accuracy and avoid the loss of significant
digits, we employed multi-precision arithmetic with 1800
significant decimal digits in the programming package of
Mathematica 8. In order to compare them with the methods
(\ref{b3}), (\ref{b5}), (\ref{b7}), (\ref{b9}), (\ref{b10}),
(\ref{b11}), (\ref{b12}) and (\ref{b14}) we choose the initial
value $x_0$ using the \texttt{Mathematica} command
\texttt{FindRoot} \cite[pp.\ 158--160]{Hazrat} and compute the
error, the computational order of convergence, (COC) by the
approximate formula \cite{coc}
\[
\textup{COC}\approx\frac{\ln|(x_{n+1}-x^{*})/(x_{n}-x^{*})|}{\ln|(x_{n}-x^{*})/(x_{n-1}-x^{*})|}.
\]\\
and the approximated computational order of convergence, (ACOC) by
the formula \cite{acoc}
\[
\textup{ACOC}\approx\frac{\ln|(x_{n+1}-x_{n})/(x_{n}-x_{n-1})|}{\ln|(x_{n}-x_{n-1})/(x_{n-1}-x_{n-2})|}.
\]\\
It is worth noting although the former formula, COC, has been used
in the recent years, nevertheless, the later, ACOC, is more
practical. Here we have collect and use both of them for checking
the accuracy of the considered methods. Moreover, we should note
that the results for these formula are generally different from
the exact convergence order of the method. The reason is that in
the error equations of the methods, we have some coefficients that
depend on $c_k$, and these $c_k$s may vanish or vary for different
kinds of examples. See the out puts in the Tables 1 and 2. We
should be careful about these events. Indeed, it does not
contradicts our discussed theory since all of the formulas are
provided approximately and behave asymptotically.

%------------------------------------------------------------------------------------------------%
\begin{center}\textbf{Table 1:} \\
$f(x)=\sin(x)-\frac{x}{100}, x^{*}=0, x_0=0.1$\\
\begin{tabular}{ c c c c c c c} \hline
Methods  & $~~~~\vert x_{1} -x^{*} \vert~~~~$ & $~~~~\vert  x_{2}
- x^{*} \vert~~~~$ & $~~~~\vert
x_{3} - x^{*} \vert~~~~$ & COC & ACOC\\
\hline
$(\ref{a111})$ & $~~0.949e-14$ &$0.486e-157~~$ & $0.314e-1733$ & $11.0000$ & $11.0000$\\
$(\ref{a112})$ & $~~0.929e-14$ &$0.387e-157~~$ & $0.252e-1734$& $11.0000$ & $11.0000$ \\
$(\ref{a113})$ & $~~0.877e-14$ &$0.204e-157~~$ & $0.223e-1737$& $11.0000$ & $11.0000$\\
$(\ref{a114})$ & $~~0.971e-14$ &$0.629e-157~~$ & $0.531e-1732$& $11.0000$ & $11.0000$\\
\hline
\end{tabular}\end{center}
\begin{center}\textbf{Table 2:} \\
$f(x)=\tan^{-1}(x), x^{*}=0, x_0=0.1$\\
\begin{tabular}{ c c c c c c c } \hline
Methods  & $~~~~\vert x_{1} -x^{*} \vert~~~~$ & $~~~~\vert  x_{2}
- x^{*} \vert~~~~$ & $~~~~\vert
x_{3} - x^{*} \vert~~~~$ &  COC & ACOC\\
\hline
$(\ref{a111})$  & $~~0.769e-12$ &$0.424e-134~~$ & $0.610e-1479$ & $11.0000$ & $11.0000$\\
$(\ref{a112})$  & $~~0.758e-12$ &$0.361e-134~~$ & $0.103e-1479$& $11.0000$ & $11.0000$ \\
$(\ref{a113})$  & $~~0.728e-12$ &$0.232e-134~~$ & $0.819e-1482$& $11.0000$ & $10.9999$ \\
$(\ref{a114})$  & $~~0.782e-12$ &$0.509e-134~~$ & $0.455e-1478$& $11.0000$ & $11.0000$\\

\hline
\end{tabular}\end{center}

In Table 1 and 2 our new three-point methods (\ref{a111}),
(\ref{a112}), (\ref{a113}) and (\ref{a114}) with weight functions
(\ref{a36}),  (\ref{a37}), (\ref{a38}) and (\ref{a39}) are tested
on two nonlinear equations.
\newpage
\begin{center}\textbf{Table 3:} \\
$f(x)=e^{\sin(x)}-1-\frac{x}{5}, x^{*}=0, x_0=0.1$\\
\begin{tabular}{ c c c c c c} \hline
Methods & $~~~~\vert x_{1} -x^{*} \vert~~~~$ & $~~~~\vert  x_{2} -
x^{*} \vert~~~~$ & $~~~~\vert
x_{3} - x^{*} \vert~~~~$ & COC & ACOC\\
\hline
$(\ref{a111})$ & $~~0.551e-9$ &$0.735e-83~~$ & $0.982e-748$ & $9.0000$ & $9.0000$\\
$(\ref{a112})$ & $~~0.346e-9$ &$0.628e-85~~$ & $0.132e-766$& $9.0000$ & $9.0000$\\
$(\ref{a113})$ & $~~0.543e-10$ &$0.362e-93~~$ & $0.943e-842$& $9.0000$ & $9.0000$\\
$(\ref{a114})$ & $~~0.768e-9$ &$0.226e-81~~$ & $0.373e-734$& $9.0000$ & $9.0000$\\
$(\ref{b3})$ & $~~0.123e-9$ &$0.162e-89~~$ & $0.181e-808$& $9.0000$ & $9.0000$\\
$(\ref{b5})$ & $~~0.266e-8$ &$0.108e-68~~$ & $0.831e-552$&
$8.0000$ & $7.9999$
\\
$(\ref{b7})$ & $~~0.589e-9$ &$0.128e-74~~$ & $0.673e-600$& $8.0000$ & $7.9999$\\
$(\ref{b9})$ &$~~0.672e-9$ &$0.199e-74~~$ & $0.119e-598$& $8.0000$ & $7.9999$\\
$(\ref{b10})$ & $~~0.125e-9$ &$0.175e-89~~$ & $0.380e-808$& $9.0000$ & $9.0000$ \\
$(\ref{b11})$ &$~~0.815e-9$ &$0.263e-73~~$ & $0.315e-598$& $8.0000$ & $7.9999$\\
$(\ref{b12})$ & $~~0.109e-7$ &$0.524e-63~~$ & $0.140e-505$& $8.0000$ & $7.9999$\\
$(\ref{b14})$ & $~~0.542e-9$ &$0.133e-74~~$ & $0.178e-599$& $8.0000$ & $8.0000$\\
\hline
\end{tabular}\end{center}

\begin{center}\textbf{Table 4:} \\
$f(x)=\ln(1-x+x^2)+4\sin(1-x), x^{*}=1, x_0=1.1$\\
\begin{tabular}{ c c c c c c} \hline
 Methods & $~~~~\vert x_{1} -x^{*}
\vert~~~~$ & $~~~~\vert  x_{2} - x^{*} \vert~~~~$ & $~~~~\vert
x_{3} - x^{*} \vert~~~~$ & COC & ACOC\\
\hline
$(\ref{a111})$  & $~~0.225e-12$ &$0.274e-117~~$ & $0.162e-1061$ & $9.0000$ & $8.9999$\\
$(\ref{a112})$ & $~~0.300e-12$ &$0.473e-116~~$ & $0.284e-1051$& $9.0000$ & $8.9999$ \\
$(\ref{a113})$ & $~~0.469e-12$ &$0.395e-114~~$ & $0.845e-1033$& $9.0000$ & $8.9999$\\
$(\ref{a114})$ & $~~0.159e-12$ &$0.966e-119~~$ & $0.577e-1075$& $9.0000$ & $8.9999$\\
$(\ref{b3})$ & $~~0.423e-12$ &$0.134e-114~~$ & $0.445e-1037$& $9.0000$ & $8.9999$\\
$(\ref{b5})$ & $~~0.295e-11$ &$0.629e-96~~$ & $0.265e-773$& $8.0000$ & $7.9999$\\
$(\ref{b7})$ & $~~0.172e-11$ &$0.581e-98~~$ & $0.984e-760$& $8.0000$ & $7.9999$\\
$(\ref{b9})$ & $~~0.357e-11$ &$0.343e-95~~$ & $0.251e-797$& $8.0000$ & $7.9999$\\
$(\ref{b10})$ & $~~0.423e-12$ &$0.135e-114~~$ & $0.474e-1037$& $9.0000$ & $8.9999$ \\
$(\ref{b11})$ & $~~0.179e-11$ &$0.839e-98~~$ & $0.195e-788$& $8.0000$ & $7.9999$\\
$(\ref{b12})$ & $~~0.829e-11$ &$0.977e-92~~$ & $0.362e-739$& $8.0000$ & $7.9999$\\
$(\ref{b14})$ &
 $~~0.211e-11$ &$0.494e-97~~$ & $0.250e-782$& $8.0000$ & $7.9999$\\
\hline
\end{tabular}\end{center}
In Tables 3 and 4 we compare our new method with the methods
(\ref{b3}), (\ref{b5}), (\ref{b7}), (\ref{b9}), (\ref{b10}),
(\ref{b11}), (\ref{b12}) and (\ref{b14}).
\newpage
\subsection{Dynamic behavior }

We already observed that all methods converge if the initial guess
is chosen suitably. We now investigate the stability region. In
other words, we numerically approximate the domain of attraction
of the zeros as a qualitative measure of stability. To answer the
important question on the dynamical behavior of the algorithms, we
investigate the dynamics of the new methods and compare them with
common and well-perfoming methods from the literature. It turns
out that only one method, namely CFGT, has better stability than
ours. In the following we recall some basic concepts such as basin
of attraction. For more details one can consult
\cite{Amat1}-\cite{Amat5}, \cite{Neta1}-\cite{Neta2},
\cite{Scott,Soleymani2,Stewart,Vrscay}.

Let $G:\C \to \C $ be a rational map on the complex plane. For
$z\in \C $, we define its orbit as the set
$orb(z)=\{z,\,G(z),\,G^2(z),\dots\}$. A point $z_0 \in \C $ is
called periodic point with minimal period $m$ if $G^m(z_0)=z_0$,
where $m$ is the smallest integer with this property. A periodic
point with minimal period $1$ is called fixed point. Moreover, a
point $z_0$ is called attracting if $|G'(z_0)|<1$, repelling if
$|G'(z_0)|>1$, and neutral otherwise. The Julia set of a nonlinear
map $G(z)$, denoted by $J(G)$, is the closure of the set of its
repelling periodic points. The complement of $J(G)$ is the Fatou
set $F(G)$, where the basin of attraction of the different roots
lie \cite{Babajee}, \cite{Cordero5}.

For the dynamical point of view, in fact, we take a $256 \times
256$ grid of the square $[-3,3]\times[-3,3]\in \C$ and assign a
color to each point $z_0\in D$ according to the simple root to
which the corresponding orbit of the iterative method starting
from $z_0$ converges, and we mark the point as black if the orbit
does not converge to a root, in the sense that after at most 100
iterations it has a distance to any of the roots, which is larger
than $10^{-3}$. In this way, we distinguish the attraction basins
by their color for different methods.

We have tested several different examples, and the results on the
performance of the tested methods were similar. Therefore we
merely report the general observation here for $f(z)=z^3-1/z$. A
visual inspection of the simulations indicates that for some
examples the SLSS method (see Fig. \ref{fig:figure1}) seems to
produce a larger basin of attraction than the BCST, SS, CTV, TP,
CL, BRW, WL methods (see Figs. \ref{fig:figure2}-\ref{fig:figure5}
and Figs. \ref{fig:figure7}-\ref{fig:figure9}), but it seems to be
smaller than that of the CFGT method (see Fig. \ref{fig:figure6}).
We stop here for a moment. Although we were able to ignore the
method CFGT, however, we should note that it is a very good
example to discuss some aspects of our algorithms. It is
well-known that any good algorithm should study these three
concepts: accuracy, efficiency, and stability. All the work in
this study have the same efficiency, four functional evaluations
per iterate. On the other hand, comparing CFGT and method (44)
reveal another fact: while a method may have a slightly better
accuracy, see Table 4 and compare numerical results for methods
(44) and (56), the other method may have produce a little better
stability. Therefore, we cannot conclude which one is better in
action. One has better accuracy, and the other has better
stability. On the whole, finding such examples could make deeper
understanding of  devising new algorithms and it can be left for
future works. Note that some points belong to no basin of
attraction; these are starting points for which the methods do not
converge, denoted by black points. These exceptional points
constitute the Julia set of methods, so named in honor of G.
Julia, a French mathematician who published an important memoir on
this subject in 1918. Here, we would like to tell a little more
about these black points. We have said that these point do not
converge to the roots. This statement is true only for the given
number of iterations, say 100 here. If we increase the number of
iteration, they might converge to a root, and the basins or Fatou
set might be larger.

\newpage
Test problem $f(z)=z^3-\frac{1}{z}$
\\
\begin{figure}[ht]
\begin{minipage}[b]{0.25\linewidth}
\centering
\includegraphics[width=\textwidth]{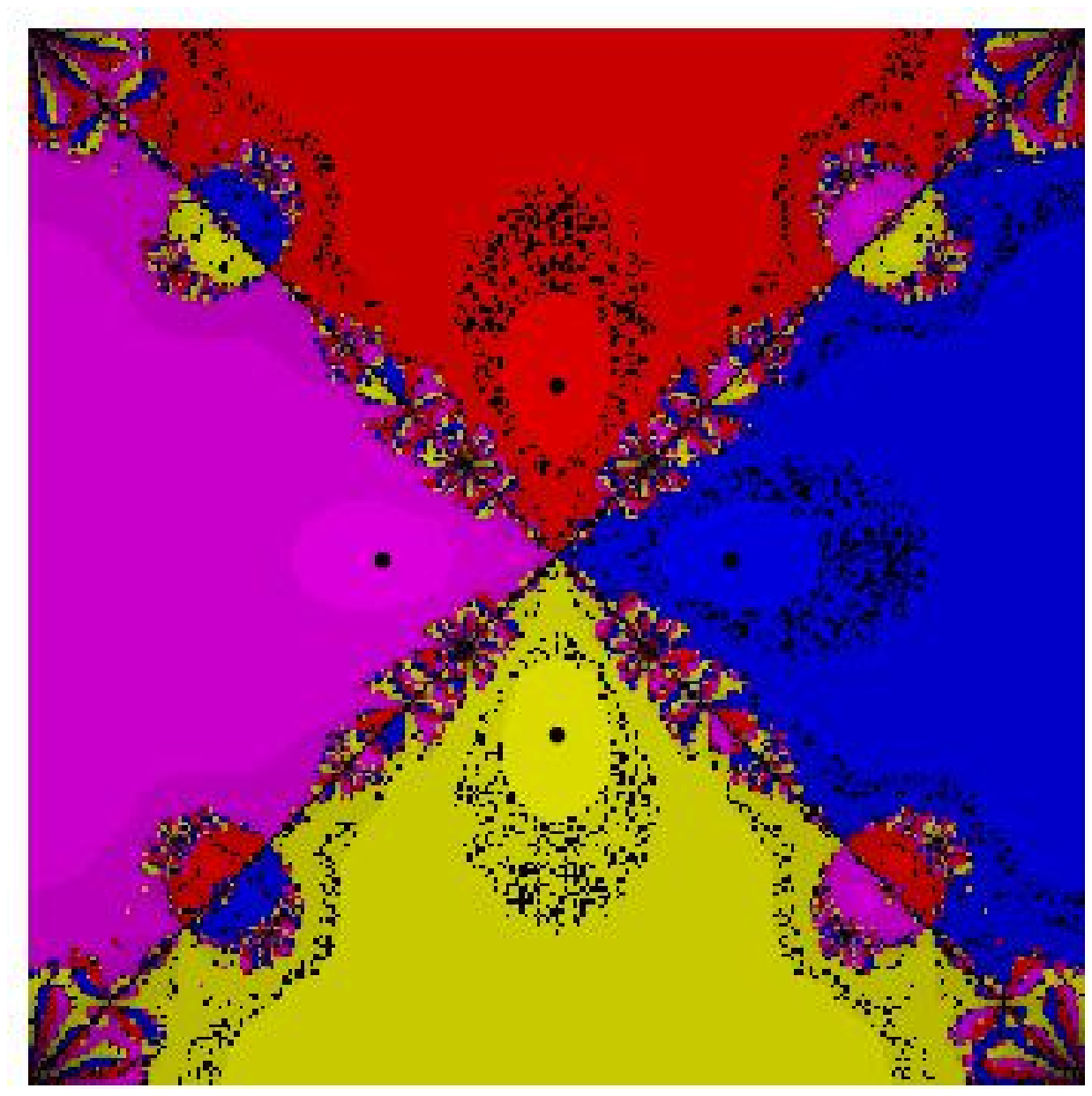}
\caption{SLSS} \label{fig:figure1}
\end{minipage}
\hspace{0.9cm}
\begin{minipage}[b]{0.25\linewidth}
\centering
\includegraphics[width=\textwidth]{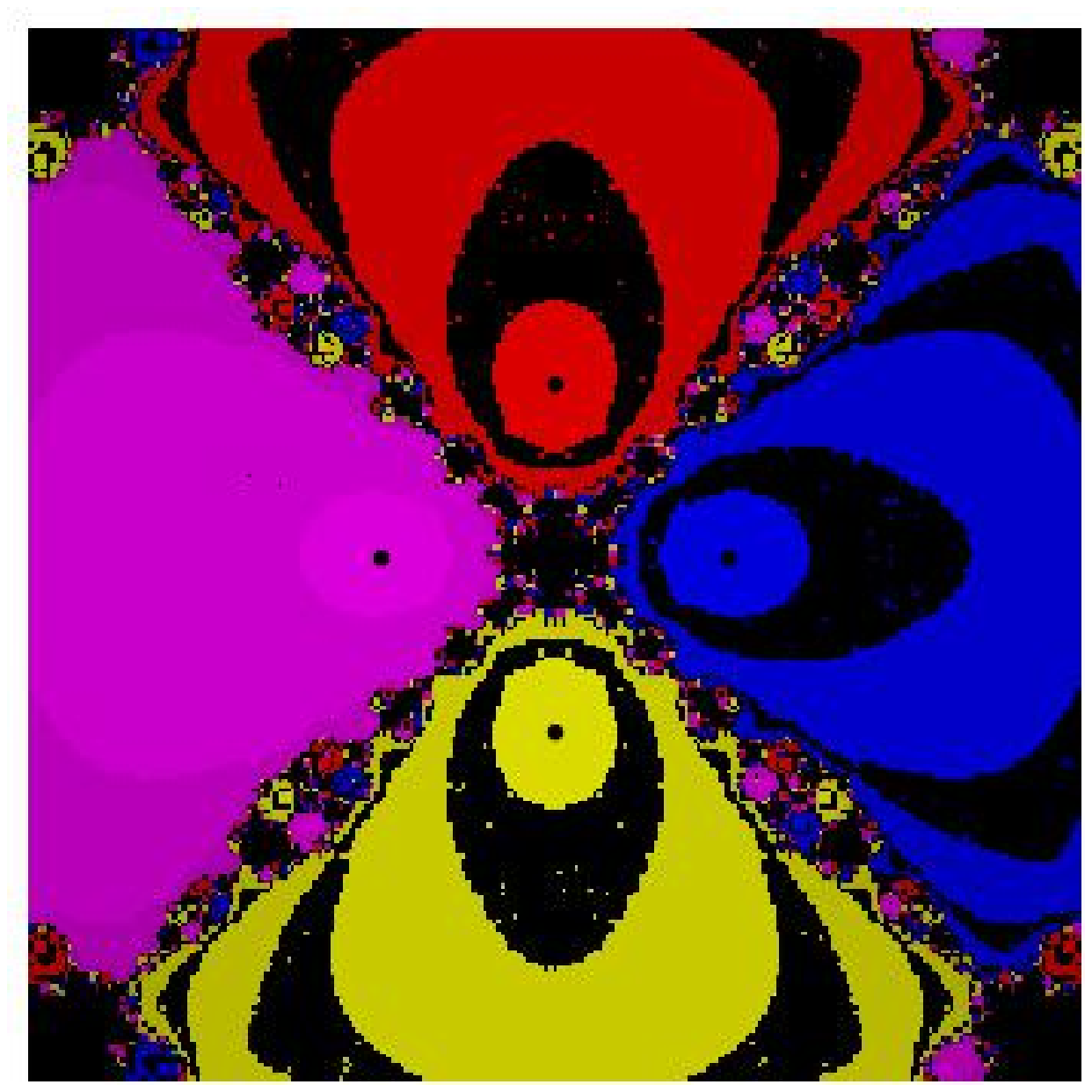}
\caption{BRW} \label{fig:figure2}
\end{minipage}
\hspace{0.9cm}
\begin{minipage}[b]{0.25\linewidth}
\centering
\includegraphics[width=\textwidth]{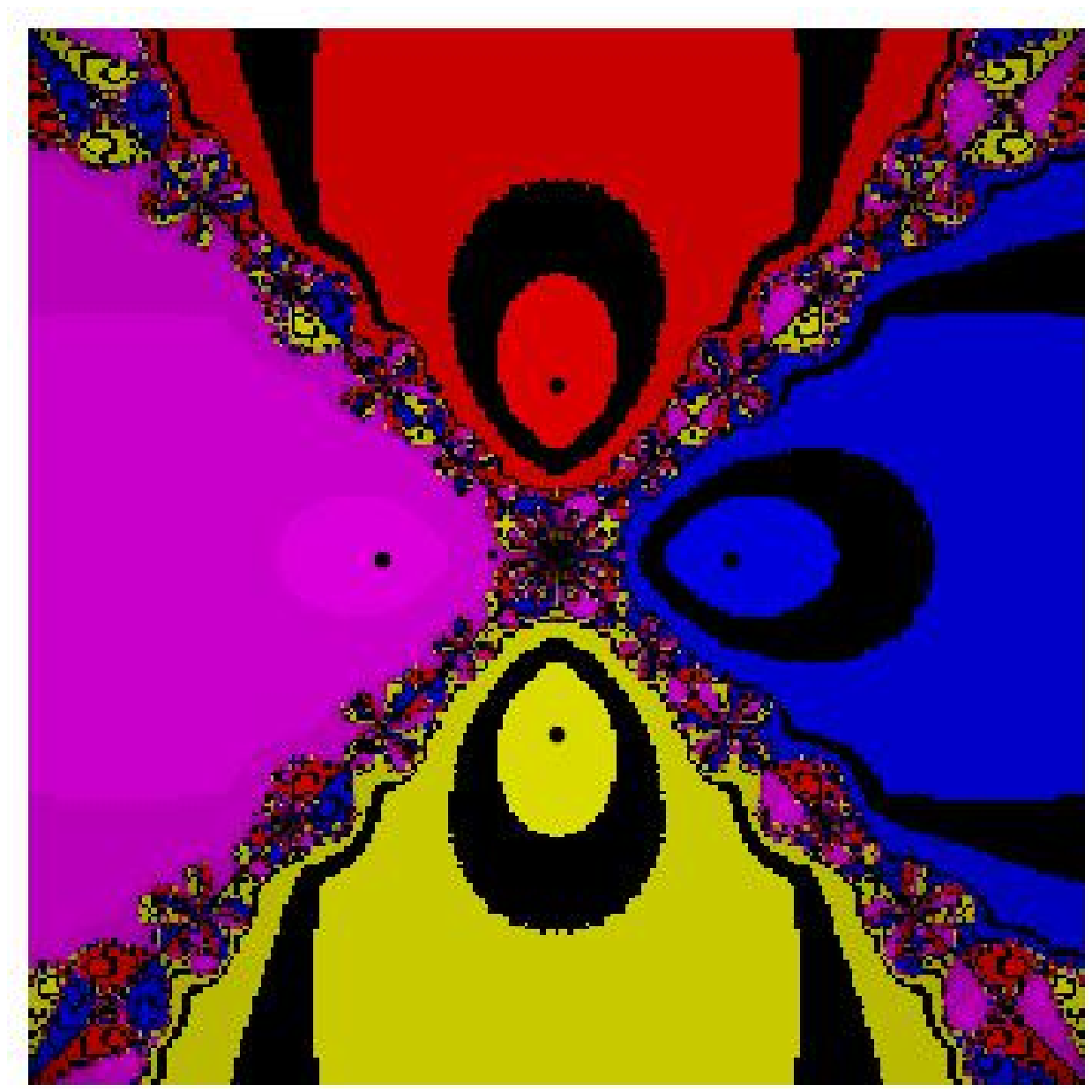}
\caption{WL} \label{fig:figure3}
\end{minipage}
\end{figure}
\\
\\
\begin{figure}[ht]
\begin{minipage}[b]{0.25\linewidth}
\centering
\includegraphics[width=\textwidth]{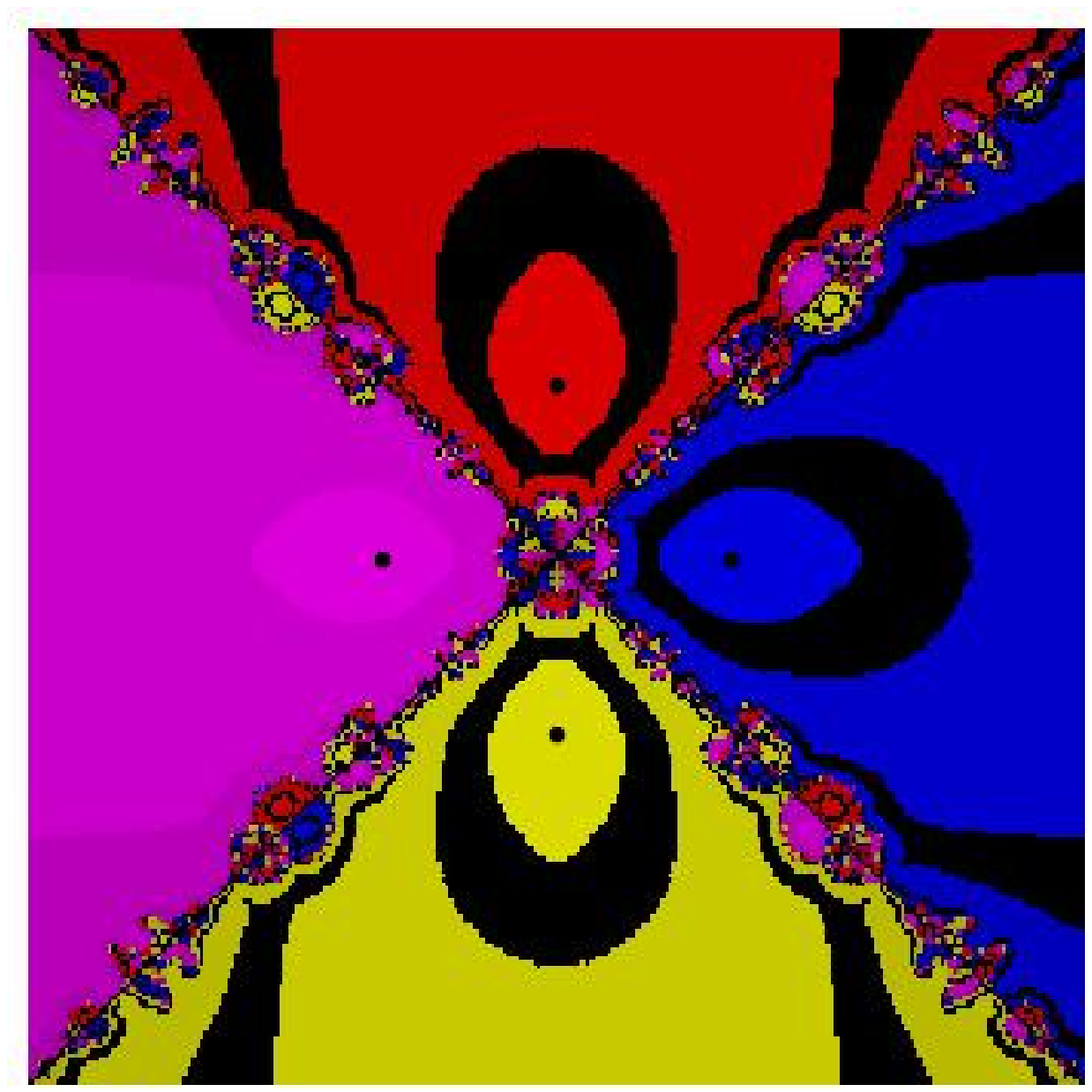}
\caption{SS} \label{fig:figure4}
\end{minipage}
\hspace{0.9cm}
\begin{minipage}[b]{0.25\linewidth}
\centering
\includegraphics[width=\textwidth]{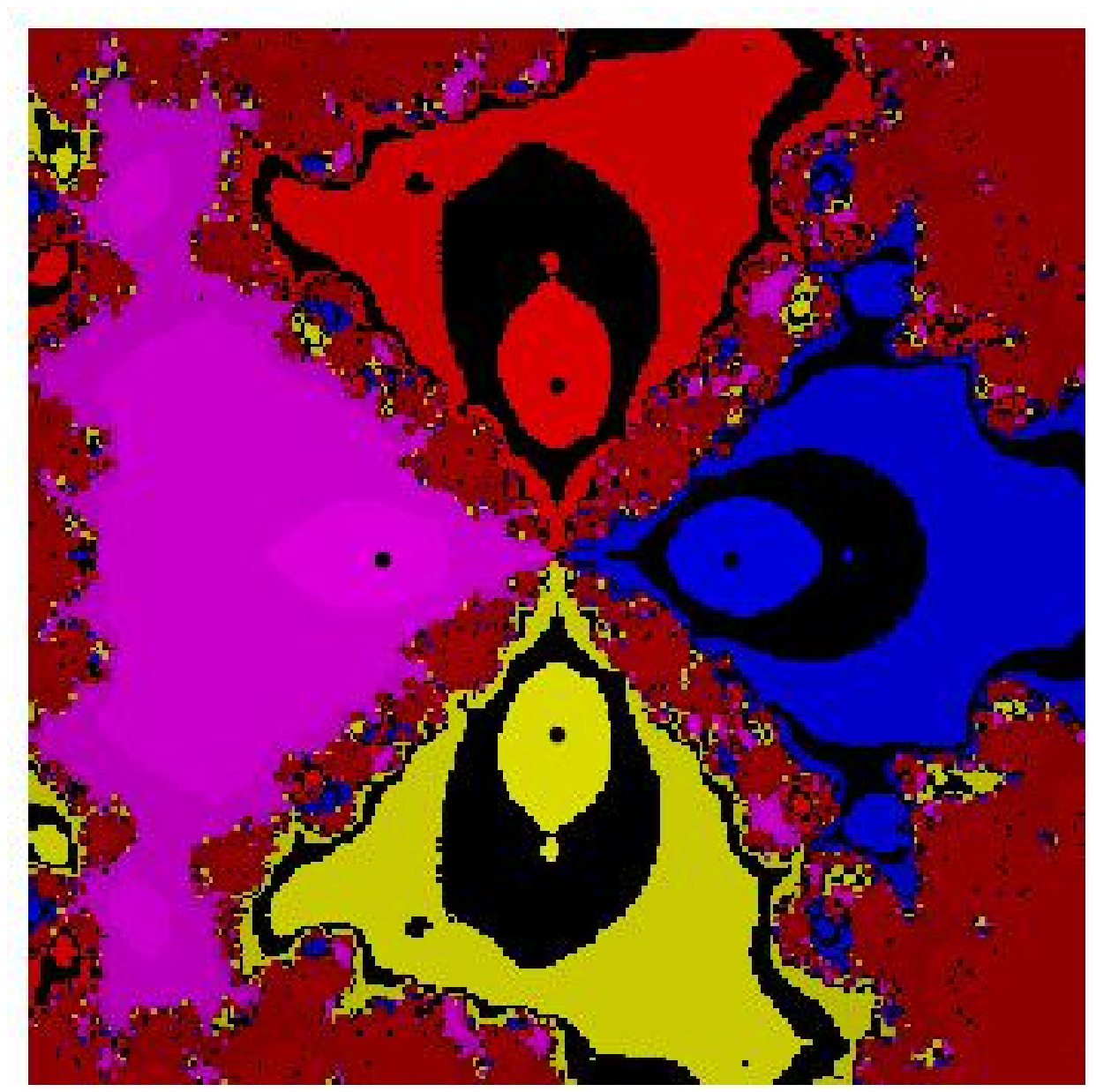}
\caption{BCST} \label{fig:figure5}
\end{minipage}
\hspace{0.9cm}
\begin{minipage}[b]{0.25\linewidth}
\centering
\includegraphics[width=\textwidth]{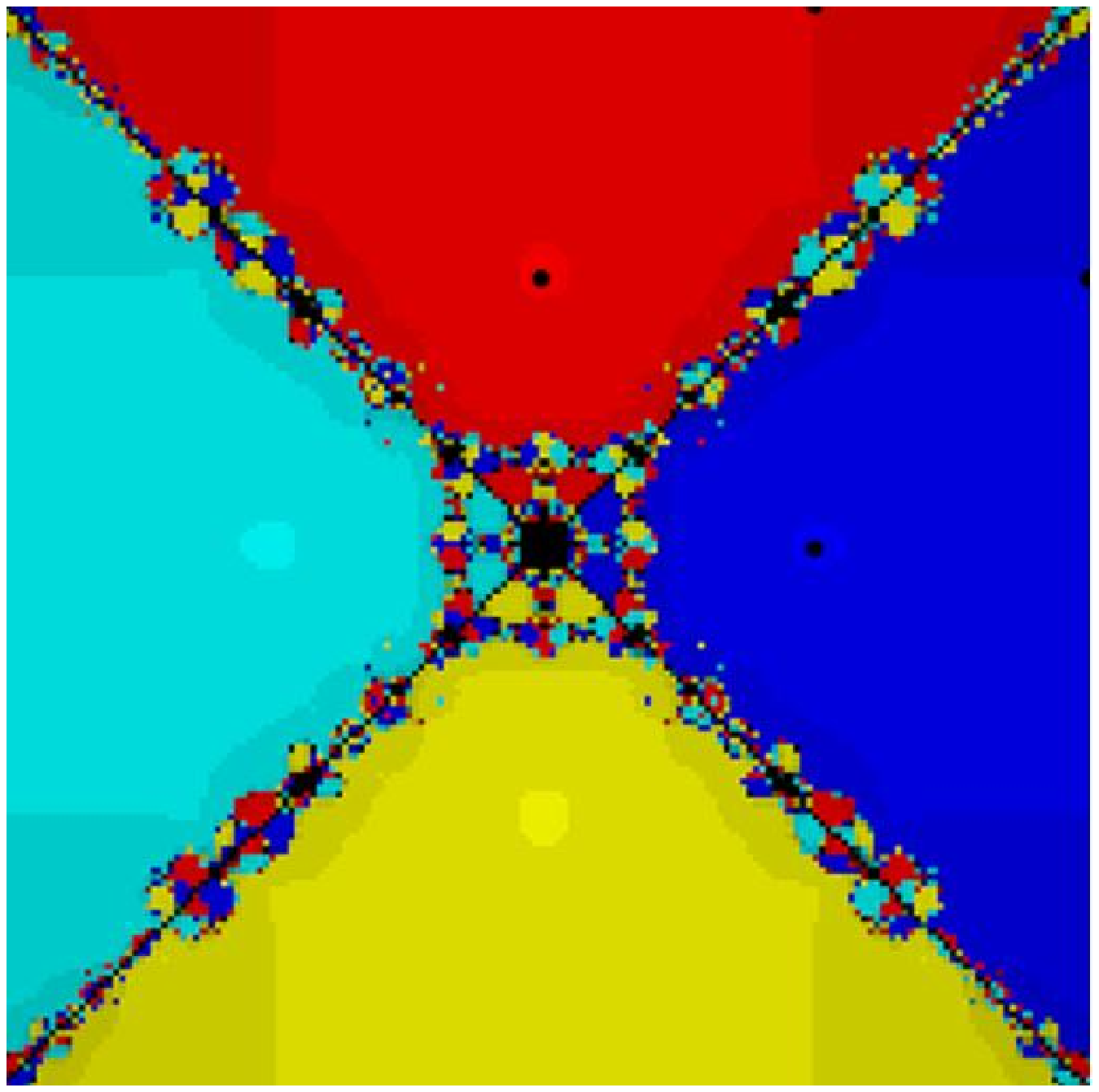}
\caption{CFGT} \label{fig:figure6}
\end{minipage}
\end{figure}
\\

\begin{figure}[ht]
\begin{minipage}[b]{0.25\linewidth}
\centering
\includegraphics[width=\textwidth]{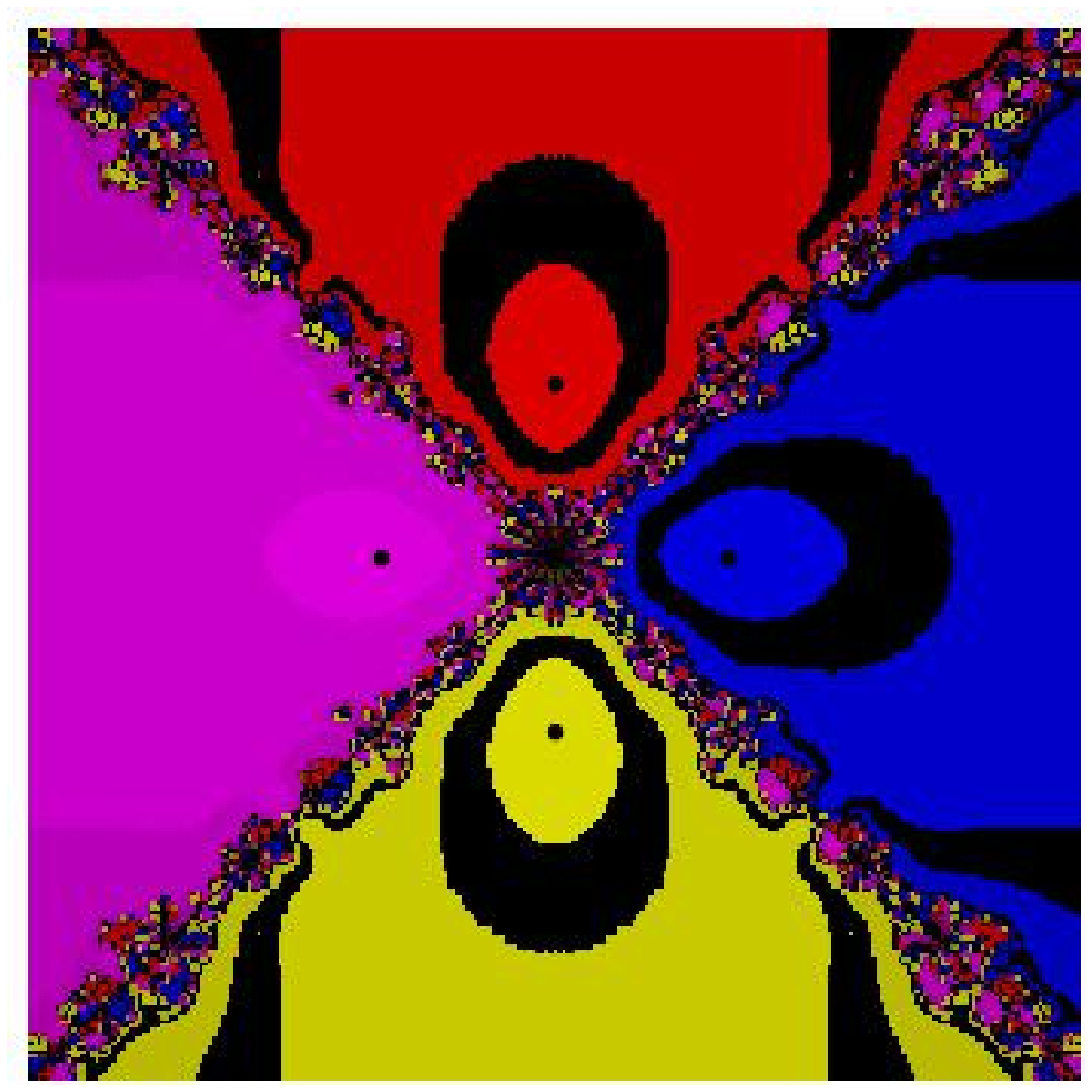}
\caption{CTV} \label{fig:figure7}
\end{minipage}
\hspace{0.9cm}
\begin{minipage}[b]{0.25\linewidth}
\centering
\includegraphics[width=\textwidth]{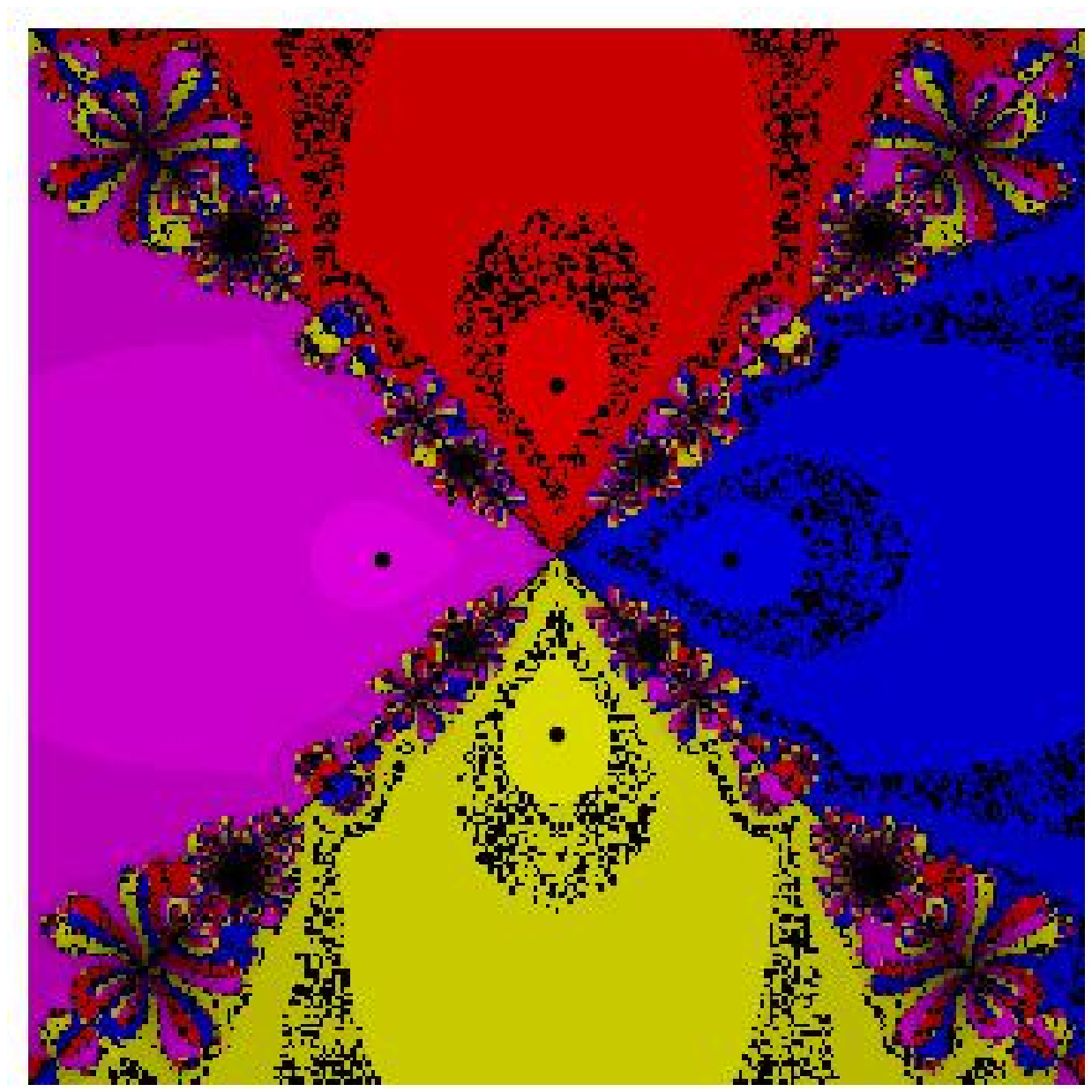}
\caption{TP} \label{fig:figure8}
\end{minipage}
\hspace{0.9cm}
\begin{minipage}[b]{0.25\linewidth}
\centering
\includegraphics[width=\textwidth]{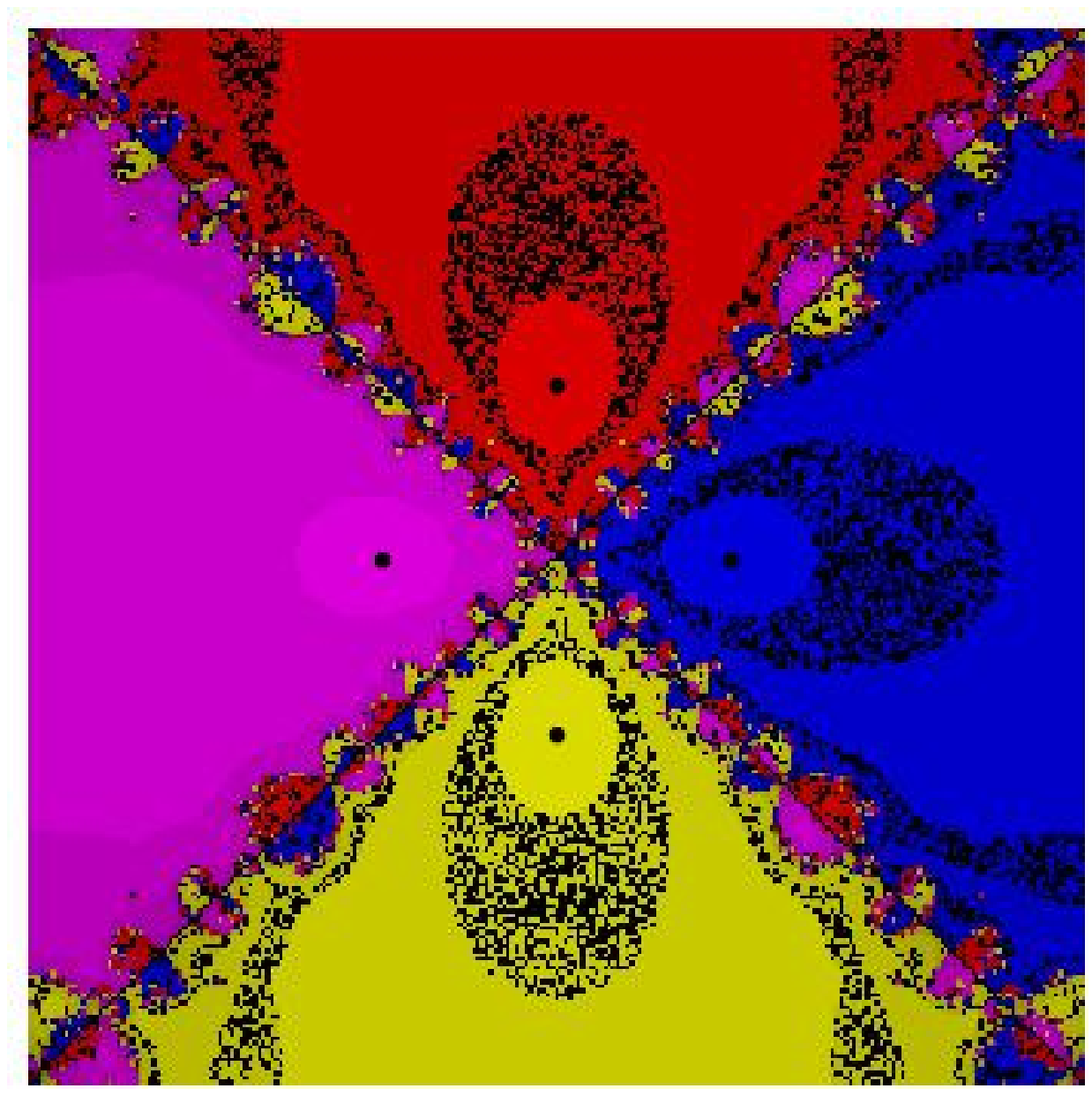}
\caption{CL} \label{fig:figure9}
\end{minipage}
\end{figure}
\newpage
\newpage
\section{Conclusion}

Two new optimal classes of two-point and three-point methods
without memory have been developed which use only three and four
function evaluations per iteration, respectively. Both methods are
based on the Newton and Secant methods. A numerical comparison
with other well-known optimal multi-point methods shows that our
new classes are a valuable alternative to existing optimal
multi-point methods. In addition, a numerical investigation of the basins of attraction of the solutions illustrate that the stability region of our method it typically larger than that of other methods. Indeed, among the eight compared methods, only one shows a larger stability region than our proposed methods.

\end{document}